\documentclass[11pt,reqno]{amsart}
\usepackage{mathrsfs}
\usepackage{amsfonts} %%% i.e. use 12pt type
\usepackage{graphicx,latexsym,bm,amsmath,amssymb,verbatim,multicol,lscape}
\newtheorem{thm}{Theorem} [section]

\newtheorem{lem}[thm]{Lemma}

\theoremstyle{definition}

\theoremstyle{remark}

\numberwithin{equation}{section}
\begin{document}
\title[symmetric functions of reciprocal arithmetic progressions]
{The elementary symmetric functions of reciprocal arithmetic progressions}
\author{Chunlin Wang and Shaofang Hong$^*$}
%    Address of record for the research reported here
\address{Mathematical College, Sichuan University, Chengdu 610064, P.R. China}
%    Current address
%\curraddr{}
\email{wdychl@126.com (C. Wang); sfhong@scu.edu.cn, hongsf02@yahoo.com,
s-f.hong@tom.com (S. Hong)}
\thanks{$^*$Hong is the corresponding author and was supported partially
by National Science Foundation of China Grant \#11371260
and by the Ph.D. Programs Foundation of Ministry of Education
of China Grant \#20100181110073.}
\date{}%
%\dedicatory{}%
%\commby{}%
% ----------------------------------------------------------------
\begin{abstract}
Let $a$ and $b$ be positive integers.
In 1946, Erd\H{o}s and Niven proved that there are only finitely
many positive integers $n$ for which one or more of the elementary
symmetric functions of $1/b, 1/(a+b),..., 1/(an-a+b)$ are integers.
In this paper, we show that for any integer $k$ with $1\le k\le n$,
the $k$-th elementary symmetric function of $1/b, 1/(a+b),..., 1/(an-a+b)$
is not an integer except that either $b=n=k=1$ and $a\ge 1$,
or $a=b=1, n=3$ and $k=2$. This strengthens the Erd\H{o}s-Niven
theorem and answers an open problem raised by Chen and Tang in 2012.
\end{abstract}

\maketitle

\section{\bf Introduction}

A well-known result in number theory states that for any integer
$n>1$, the harmonic sum $\sum_{i=1}^n\frac{1}{n}$ is not an integer.
Let $a$ and $b$ be positive integers. In 1946, Erd\H{o}s and Niven
\cite{[EN]} proved that there are only finitely
many positive integers $n$ for which one or more of the elementary
symmetric functions of $1/b, 1/(a+b),..., 1/(an-a+b)$ are integers.
Chen and Tang \cite{[CT]} proved that none of
the elementary symmetric functions of $1,1/2, ..., 1/n$ is an
integer if $n\ge 4$. Wang and Hong \cite{[WH]} proved that
none of the elementary symmetric functions of $1,1/3, ...,
1/(2n-1)$ is an integer if $n\ge 2$.

In this paper, we address the problem of determining all the finite
arithmetic progressions $\{b+ai\}_{i=0}^{n-1}$ such that one or more
elementary symmetric functions of $1/b, 1/(a+b), ..., 1/(an-a+b)$ are
integers (see Problem 1 of \cite{[CT]}). For any integer $k$ with
$1\le k\le n$, let $S_{a,b}(n,k)$ denote the $k$-th elementary
symmetric function of $1/b, 1/(a+b), ..., 1/(an-a+b)$. That is,
$$
S_{a,b}(n,k):=\sum_{0\le i_1<...<i_k\le
n-1}\prod_{j=1}^k\frac{1}{ai_j+b}.
$$
In the present paper, we introduce a new method to investigate the above question.
In particular, by improving greatly the arguments in \cite{[CT]}, \cite{[EN]} and
\cite{[WH]}, and providing a detailed analysis to $S_{a,b}(n,k)$,
we show the following result.
\begin{thm}
Let $a, b, n$ and $k$ be positive integers with $1\le k\le n$.
Then $S_{a,b}(n,k)$ is not an integer except that either
$b=n=k=1$, or $a=b=1, n=3$ and $k=2$, in
which case $S_{a,b}(n,k)$ is an integer.
\end{thm}

Clearly, Theorem 1.1 strengthens the Erd\H{o}s-Niven theorem
and answers completely Problem 1 of \cite{[CT]}. The key tool of
the current paper is to use an effective result of Dusart \cite{[Du2]}
on the distribution of primes, see Lemma 2.3 below.

The paper is organized as follows. First, in
Section 2, we show several lemmas which are needed for the proof of
Theorem 1.1. Finally, in Section 3, we give the details of
the proof of Theorem 1.1.

As usual, we denote by $\lfloor x\rfloor$ and $\lceil x\rceil$
the biggest integer no more than $x$ and the smallest integer no
less than $x$, respectively. Let $v_p$ denote the $p$-adic
valuation on the field $\mathbb{Q}$ of rational numbers,
i.e., $v_p(a)=b$ if $p^{b}$ divides $a$ and $p^{b+1}$
does not divide $a$.

\section{\bf Preliminary lemmas}

In this section, we show some preliminary lemmas, which are needed
for the proof of Theorem 1.1. We begin with the following
well-known result.

\begin{lem} \cite{[EN]} \cite{[Na]}
Let $a,b,n$ and $k$ be positive integers with $n\ge 2$.
Then $S_{a,b}(n,1)$ is not an integer.
\end{lem}

\begin{lem}
Let $a,b,n$ and $k$ be positive integers such that $2\le k\le n$. If either
$n\le\frac{b}{a}\big(e^{a(\sqrt{2b^2+1}-1)/b}-1\big)$ or
$k\ge\frac{e}{a}\log\frac{an+b}{b}+\frac{e}{b},$ then $0<S_{a,b}(n,k)<1$.
\end{lem}
\begin{proof}
Evidently, $S_{a,b}(n,k)>0$. It remains to show that $S_{a,b}(n,k)<1$.
If $k=n$, it is easy to see that $S_{a,b}(n,k)<1$. In the following, we
assume that $2\le k\le n-1$.

First, we let $n\le\frac{b}{a}e^{a(\sqrt{2b^2+1}-1)/b}-\frac{b}{a}$.
By the multinomial expansion theorem, we deduce that
\begin{align*}
S_{a,b}(n,k)=&\sum\limits_{0\le i_1<...<i_k\le n-1}\prod\limits_{j=1}^k\frac{1}{ai_j+b}\\
=&\sum\limits_{0=i_1< i_2<...<i_k\le n-1}\frac{1}{b}\prod\limits_{j=2}^{k}\frac{1}{ai_j+b}
+\sum\limits_{1\le i_1<...<i_k\le n-1}\prod\limits_{j=1}^k\frac{1}{ai_j+b}\\
\le&\frac{1}{b(k-1)!}\Big(\sum\limits_{i=1}^{n-1}\frac{1}{ai+b}\Big)^{k-1}
+\frac{1}{k!}\Big(\sum\limits_{i=1}^{n-1}\frac{1}{ai+b}\Big)^k. \tag{2.1}
\end{align*}
Since
$$
\sum_{i=1}^{n-1}\frac{1}{ai+b}<\sum_{i=1}^{n-1}\int_{i-1}^i\frac{1}{ax+b}dx
<\frac{1}{a}\log\frac{an+b}{b}, \eqno(2.2)
$$
by (2.1) one concludes that
\begin{align*}
S_{a,b}(n,k)<\frac{1}{b(k-1)!}\Big(\frac{1}{a}\log\frac{an+b}{b}\Big)^{k-1}
+\frac{1}{k!}\Big(\frac{1}{a}\log\frac{an+b}{b}\Big)^k. \tag{2.3}
\end{align*}

But $n\le\frac{b}{a}e^{a(\sqrt{2b^2+1}-1)/b}-\frac{b}{a}$ gives us that
$\frac{1}{a}\log\frac{an+b}{b}\le \frac{\sqrt{2b^2+1}-1}{b}$.
It then follows from (2.3) and $k\ge 2$ that
\begin{align*}
&S_{a,b}(n,k)<\frac{1}{b(k-1)!}\Big(\frac{\sqrt{2b^2+1}-1}{b}\Big)^{k-1}
+\frac{1}{k!}\Big(\frac{\sqrt{2b^2+1}-1}{b}\Big)^k \\
=&\frac{\sqrt{2b^2+1}-1}{b^2}\frac{\Big(\frac{\sqrt{2b^2+1}-1}{b}\Big)^{k-2}}{(k-1)!}
+\frac{(\sqrt{2b^2+1}-1)^2}{2b^2}\frac{2\Big(\frac{\sqrt{2b^2+1}-1}{b}\Big)^{k-2}}{k!}.
\tag{2.4}
\end{align*}
Obviously $0<\frac{\sqrt{2b^2+1}-1}{b}<\sqrt{2}$
for all positive integers $b$. Hence for all $k\ge 2$, we have
$$0<\frac{2\Big(\frac{\sqrt{2b^2+1}-1}{b}\Big)^{k-2}}{k!}
\le\frac{\Big(\frac{\sqrt{2b^2+1}-1}{b}\Big)^{k-2}}{(k-1)!}\le 1.\eqno(2.5)
$$
Then (2.4) together with (2.5) infers that
$$S_{a,b}(n,k)<\frac{\sqrt{2b^2+1}-1}{b^2}+\frac{(\sqrt{2b^2+1}-1)^2}{2b^2}=1$$
as desired. So Lemma 2.2 is true if
$n\le\frac{b}{a}e^{a(\sqrt{2b^2+1}-1)/b}-\frac{b}{a}$.

Consequently, let $k\ge\frac{e}{a}\log\frac{an+b}{b}+\frac{e}{b}$.
The multinomial expansion theorem together with (2.2) tells us that
$$S_{a,b}(n,k)\le\frac{1}{k!}\Big(\sum\limits_{i=0}^{n-1}\frac{1}{b+ai}\Big)^k
<\frac{1}{k!}\bigg(\frac{1}{b}+\frac{1}{a}\log\frac{an+b}{b}\bigg)^k. \eqno (2.6)$$
On the other hand, since $k\ge\frac{e}{a}\log\frac{an+b}{b}+\frac{e}{b}$, one has
\begin{align*}
\log k!=&\sum_{i=2}^{k}\log i>\int_{1}^{k}\log x {\rm d}x>k(\log k-1) \\
\ge& k\log\bigg(\frac{1}{b}+\frac{1}{a}\log\frac{an+b}{b}\bigg)
=\log\bigg(\frac{1}{b}+\frac{1}{a}\log\frac{an+b}{b}\bigg)^k,
\end{align*}
which implies that the right-hand side of (2.6) is strictly less than 1.
So (2.6) concludes the desired result $S_{a,b}(n,k)<1.$
Thus Lemma 2.2 is proved in this case.

The proof of Lemma 2.2 is complete.
\end{proof}

\begin{lem} \cite{[Du1]} \cite{[Du2]} For any real number $x\ge 3275$,
there is a prime number $p$ satisfying that $x<p\le x(1+\frac{1}{2\log^2x}).$
\end{lem}

\begin{lem}
Let $a$ and $b$ be positive integers. Let $n$ be an integer
satisfying that $n>120000$ if $a\le 18$ and $b\le\frac{3275a(\sqrt{2}e+1)}{e^a-1}+1$,
and $n>\frac{b}{a}\big(e^{a(\sqrt{2b^2+1}-1)/b}-1\big)$ otherwise.
Then for any integer $k$ with $1\le k<\frac{e}{a}\log\frac{an+b}{b}+\frac{e}{b}$,
there is a prime $p$ such that $\frac{n}{k+1}<p\le \frac{n}{k}$ and $p>ak+2a+6.$
\end{lem}
\begin{proof}
To show that there is a prime $p$ such that $\frac{n}{k+1}<p\le \frac{n}{k}$,
it suffices to show that the following two inequalities hold:
$$\frac{n}{k+1}\ge 3275 \eqno(2.7)$$
and
$$2\log^2\frac{n}{k+1}\ge k.\eqno(2.8)$$
Actually, if (2.7) and (2.8) are both true, then by Lemma 2.3,
there is a prime $p$ such that
$$
\frac{n}{k+1}<p\le\frac{n}{k+1}\big(1+\frac{1}{2\log^2(n/(k+1))}\big)
\le\frac{n}{k+1}\big(1+\frac{1}{k}\big)=\frac{n}{k}.
$$
So we need only to show that (2.7) and (2.8) hold, which will be done
in what follows.

First we prove that (2.7) is true. To do so, let
$$f(x)=x-\frac{3275e}{a}\log\frac{ax+b}{b}-3275(\frac{e}{b}+1).$$
Then $f'(x)=1-3275e/(ax+b)$. Further, $f'(x)>0$ if $ax+b>3275e$.
To prove (2.7), it is sufficient to show that $f(n)>0$.
Actually, let $\Delta :=\frac{e}{a}\log\frac{an+b}{b}
+\frac{e}{b}$. Since $\Delta >k$, we have
$f(n)-n=-3275(\Delta +1)<-3275(k+1)$. So $-n<-3275(k+1)$ and
(2.7) follows immediately. It remains to show that $f(n)>0$.
We divide the proof into the following three cases:

{\sc Case 1.} $a\le 18$, $b\le\frac{3275a(\sqrt{2}e+1)}{e^a-1}+1$
and $n>120000$. Clearly, $f(120000)>0$. So $f(n)>0$ as desired.

{\sc Case 2.} $a\le 18$, $b>\frac{3275a(\sqrt{2}e+1)}{e^a-1}+1$
and $n>\frac{b}{a}\big(e^{a(\sqrt{2b^2+1}-1)/b}-1\big)$.
Since $a\ge 1$ and $(\sqrt{2b^2+1}-1)/b\ge 1$ for all integers $b>1$, then
for any real number $x>\frac{b}{a}\big(e^{a(\sqrt{2b^2+1}-1)/b}-1\big)$, we have
$$ax+b>be^{a(\sqrt{2b^2+1}-1)/b}\ge\Big(\frac{3275a(\sqrt{2}e+1)}{e^a-1}+1\Big)e^a
>3275(\sqrt{2}e+1)\frac{ae^a}{e^a-1}>3275e,$$
which implies that $f'(x)>0$. Since $b\ge\frac{3275a(\sqrt{2}e+1)}{e^a-1}+1$
and $\frac{\sqrt{2b^2+1}-1}{b}<\sqrt{2}$, one deduces that
\begin{align*}
f(n)&>f\Big(\frac{b}{a}\big(e^{a(\sqrt{2b^2+1}-1)/b}-1\big)\Big)\\
&=\frac{b}{a}\big(e^{a(\sqrt{2b^2+1}-1)/b}-1\big)
-\frac{3275e(\sqrt{2b^2+1}-1)}{b}-3275(\frac{e}{b}+1)\\
&>3275(\sqrt{2}e+1)+\frac{e^a-1}{a}-3275\sqrt{2}e
-\frac{e(e^a-1)}{a(\sqrt{2}e+1)}-3275>0
\end{align*}
as required.

{\sc Case 3.} $a>18$ and $n>\frac{b}{a}\big(e^{a(\sqrt{2b^2+1}-1)/b}-1\big)$.
Since $b\ge1$ and $(\sqrt{2b^2+1}-1)/b\ge\sqrt{3}-1$ for all positive integers $b$,
it follows that for any real number $x>\frac{b}{a}\big(e^{a(\sqrt{2b^2+1}-1)/b}-1\big)$,
we have $ax+b>be^{a(\sqrt{2b^2+1}-1)/b}\ge e^{19(\sqrt{3}-1)}>3275e$ and so $f'(x)>0$.
Hence
\begin{align*}
f(n)>f\Big(\frac{b}{a}\big(e^{a(\sqrt{2b^2+1}-1)/b}-1\big)\Big)
>\frac{e^{19(\sqrt{3}-1)}-1}{19}-3275\sqrt{2}e-3275(e+1)>0.
\end{align*}
So (2.7) is proved.

Second, we show that (2.8) is true. Since
$k<\frac{e}{a}\log\frac{an+b}{b}+\frac{e}{b}<\frac{e}{a}(\log an+1)
+\frac{e}{b}$, so to show that (2.8) holds, we need only to show
$$2\big(\log n-\log(k+1)\big)^{2}>\frac{e}{a}\log an+\frac{e}{a}+\frac{e}{b}.\eqno(2.9)$$
To show that (2.9) is true, it is enough to prove that
$$
\log n-2\log(k+1)>\frac{e}{2a}+\frac{e}{2\log n}
\Big(\frac{\log a}{a}+\frac{1}{a}+\frac{1}{b}\Big),\eqno(2.10)
$$
which will be done in what follows. Let
$$g(x)=x-2\log\Big(\frac{ex}{a}+\frac{e\log a}{a}+\frac{e}{a}+\frac{e}{b}+1 \Big)-2.$$
Then $g'(x)=1-2/(x+\log a+a/b+a/e+1)$. Since $\frac{\log x}{x}\le\frac{1}{e}$ for any real
number $x\ge1$, we have $g(9)\ge9-2\log\big(\frac{9e}{a}+1+\frac{e}{a}+\frac{e}{b}+1\big)-2
\ge7-2\log(9e+e+e+2)>0$. One can easily check that $g'(x)>0$ if $x\ge 9$.
Therefore $g(x)>0$ if $x>9$. Under the assumption, one
can conclude that $n>e^9$, that is $\log n>9$.
It then follows that
$$\frac{e}{2a}+\frac{e}{2\log n}\Big(\frac{\log a}{a}+\frac1a+\frac{1}{b}\Big)<2 \eqno(2.11)$$
Deduced from (2.11) and $k<\frac{e}{a}(\log an+1)+\frac{e}{b}$, one has
\begin{align*}
\log n-2\log(k+1)-2>&\log n-2\log\Big(\frac{e\log an}{a}
+\frac{e}{a}+\frac{e}{b}+1\Big)-2 \\
=&g(\log n)>0,
\end{align*}
which means (2.10) is true. Thereby (2.8) is proved. This concludes that there is a
prime $p$ such that $\frac{n}{k+1}<p\le \frac{n}{k}.$

Finally, we show that for any prime $p$ with $p>\frac{n}{k+1}$,
one has $p>ak+2a+6$. To do so, we need only to show that
$\frac{n}{k+1}>ak+2a+6$. Let
$$
h(x)=x-\Big(\frac{e}{a}\log\frac{ax+b}{b}+\frac{e}{b}+1\Big)
\Big(e\log\frac{ax+b}{b}+\frac{ae}{b}+2a+6\Big).
$$
Then
$$h'(x)=1-\frac{2e^2}{ax+b}\log\frac{ax+b}{b}-\frac{2ae^2+3abe+6be}{b(ax+b)}.$$
Since $k<\frac{e}{a}\log\frac{an+b}{b}+\frac{e}{b}$, we have
$n-(k+1)(ak+2a+6)>h(n)$. So, to prove that $\frac{n}{k+1}>ak+2a+6$,
we only need to show that $h(n)>0$, which we will do in the following.

If $a\le 18$ and $b\le\frac{3275a(\sqrt{2}e+1)}{e^a-1}+1$,
then for any real number $x>120000$, one has
\begin{align*}
h'(x)&\ge 1-2e^2\frac{\log(ax+b)}{ax+b}-\frac{2e^2a+3ae+6e}{ax+b}\tag{2.12}\\
&> 1-2e^2\frac{\log120000}{120000}-e\frac{2ea+3a+6}{120000}\\
&>1-\frac{18\times 21}{120000}-\frac{3(9\times 18+6)}{120000}>0.
\end{align*}
This implies that $h(n)>h(120000)>0$ for any integer $n>120000$.

If $a\le 18$ and $b>\frac{3275a(\sqrt{2}e+1)}{e^a-1}+1$, since
$1\le (\sqrt{2b^2+1}-1)/b<\sqrt{2}$ for all $b\ge2$, it then follows that
for any real number $x>\frac{b}{a}\big(e^{a(\sqrt{2b^2+1}-1)/b}-1\big)$,
we have
\begin{align*}
h'(x)&> 1-\frac{2e^2}{b}\frac{\log\frac{ax+b}{b}}
{\frac{ax+b}{b}}-\frac{e(2ea+3a+6)}{ax+b}\\
&>1-\frac{2e^2}{b}\frac{a}{e^a}-\frac{e(2ea+3a+6)}{be^a}\\
&>1-\frac{4e^2+3e+\frac{6e}{a}}{\frac{e^a}{e^a-1}\big(3275(\sqrt{2}e+1)+\frac{e^a-1}{a}\big)}\\
&>1-\frac{4e^2+9e}{3275(\sqrt{2}e+1)}>0.
\end{align*}
Hence for any integer $n$ with $n>\frac{b}{a}\big(e^{a(\sqrt{2b^2+1}-1)/b}-1\big)$, one has
\begin{align*}
&h(n)>h\Big(\frac{b}{a}\big(e^{a(\sqrt{2b^2+1}-1)/b}-1\big)\Big) \\
&>3275(\sqrt2e+1)+(e^a-1)/a-\big(\sqrt{2}e+e+1\big)\big(\sqrt{2}ae+ae+2a+6\big)>0.
\end{align*}

If $a>18$, then by (2.12) and noting that $\frac{\sqrt{2b^2+1}-1}{b}\ge \sqrt{3}-1$
for any positive integer $b$, it follows that for any real number
$x>\frac{b}{a}\big(e^{a(\sqrt{2b^2+1}-1)/b}-1\big)$, one has
\begin{align*}
h'(x)&>1-2e^2\frac{(\sqrt{3}-1)a}{e^{(\sqrt{3}-1)a}}-\frac{e(2ea+3a+6)}{e^{(\sqrt{3}-1)a}}\\
&=1-\frac{a}{e^{(\sqrt{3}-1)a}}\big(2\sqrt{3}e^2+3e+\frac{6e}{a}\big)\\
&>1-\frac{19}{e^{19(\sqrt{3}-1)}}\big(2\sqrt{3}e^2+3e+\frac{6e}{19}\big)>0.
\end{align*}
It then follows from the hypothesis
$n>\frac{b}{a}\big(e^{a(\sqrt{2b^2+1}-1)/b}-1\big)$ that
\begin{align*}
h(n)&>h\Big(\frac{b}{a}\big(e^{a(\sqrt{2b^2+1}-1)/b}-1\big)\Big) \\
&>\big(e^{a(\sqrt{3}-1}-1\big)/a-\big(\sqrt{2}e+e+1\big)\big(\sqrt{2}ae+ae+2a+6\big)>0.
\end{align*}
By the above discussion, we can now conclude that $h(n)>0$. Hence one gets that
$p>\frac{n}{k+1}>ak+2a+6$ as desired.

This completes the proof of Lemma 2.4.
\end{proof}

Now we consider the $p$-adic valuation of $S_{a,b}(n,k)$ for
the prime $p$ appeared in Lemma 2.4.

\begin{lem}
Let $a,b,n$ and $k$ be positive integers such that $1\le k\le n$.
If there is a prime $p$ satisfying that $\frac{n}{k+1}<p\le \frac{n}{k}$
and $p>ak+2a+\frac{2b}{p}$, then one has $v_p(S_{a,b}(n,k))=-k$.
\end{lem}

\begin{proof}
Since $p>ak+2a+\frac{2b}{p}$, $a$ and $p$ are relatively prime.
Hence there is a unique integer $r\in\{0,1,2,...,p-1\}$ such that
$p|(ar+b)$. Let $a_0=\frac{ar+b}{p}.$ Then $a_0<a+\frac{b}{p}$.
Evidently, one can split the sum $S_{a,b}(n,k)$ into two parts:
$S_{a,b}(n,k)=S_1+S_2,$ where
$$
S_1=\sum_{\substack{0\leq i_1<\cdots<i_k\leq n-1 \\
p\mid(ai_j+b) \ \forall 1\le j\le k}}\prod_{j=1}^{k}\frac{1}{ai_j+b} \ {\rm and} \
S_2=\sum_{\substack{0\leq i_1<\cdots<i_k\leq n-1 \\
\exists j\ {\rm s.t.\ }p\nmid(ai_j+b)}}\prod_{j=1}^{k}\frac{1}{ai_j+b}.
$$

First we rewrite the sum $S_1$. Let $ai'+b$ be any term divided by $p$
in $\{ai+b\}_{i=0}^{n-1}$. Then $ai'+b\equiv 0\pmod p$ and $0\le i'\le n-1$. Since
$ar+b\equiv 0\pmod p$, it follows that $i'\equiv r\pmod p$, i.e. $i'=r+pi''$
with $0\le i''\le \big\lfloor\frac{n-r-1}{p}\big\rfloor$.
But $\frac{n}{k+1}<p\le\frac{n}{k}$ implies that
$$k-\frac{r+1}{p}\le \frac{n-r-1}{p}<k+1-\frac{r+1}{p}.$$
Then one deduces immediately that
$\big\lfloor\frac{n-r-1}{p}\big\rfloor=k+t$, where
\begin{align*}
t=\bigg\{\begin{array}{cl}
       -1, & {\rm if}\  \frac{n-1-r}{p}<k\\
       0, & {\rm otherwise}
     \end{array}
=\bigg\{\begin{array}{cl}
       -1, & {\rm if}\  p(ak+a_0)>a(n-1)+b\\
       0, & {\rm otherwise.}
     \end{array}
\end{align*}
Thus the set of all the terms divided by $p$ in $\{ai+b\}_{i=0}^{n-1}$
is given as follows:
$$\{b+ar, b+ar+ap,..., b+ar+ap(k+t)\}$$
$$
=\{pa_0, pa_0+pa,..., pa_0+pa(k+t)\}.$$
Therefore one can rewrite the sum $S_1$ as follows:
$$
S_1=\sum_{0\leq l_1<\cdots<l_k\leq k+t}\prod_{j=1}^{k}\frac{1}{p(al_j+a_0)}
=\sum_{0\leq i_1<\cdots<i_k\leq k+t}\frac{1}{p^k}\prod_{j=1}^{k}\frac{1}{ai_j+a_0}.
\eqno(2.13)
$$

Consequently, we calculate $v_p(S_1)$. Claim that $v_p(S_1)=-k$.
In fact, since $p>ak+2a+\frac{2b}{p}>ak+2a_0$, one has $v_p(ai+a_0)=0$
for all $0\le i\le k$. Hence if $t=-1$, then by (2.13)
\begin{align*}
v_p(S_1)=-k+v_p\bigg(\prod_{i=0}^{k-1}\frac{1}{ai+a_0}\bigg)=-k.
\end{align*}
If $t=0$, then (2.13) gives us that

\begin{align*}
v_p(S_1)=&v_p\bigg(\sum_{0\leq i_1<\cdots<i_k\leq k}\frac{1}{p^k}
\prod_{j=1}^{k}\frac{1}{a_0+ai_j}\bigg)\\
=&-k+v_p\bigg(\sum_{0\le i_1<\cdots<i_k\le k}
\prod_{j=1}^{k}\frac{1}{a_0+ai_j}\bigg)                       \\
=&-k+v_p\bigg(\frac{\sum_{i=0}^{k}(a_0+ai)}{\prod_{i=0}^{k}(a_0+ai)}\bigg).
\end{align*}
But the identity
$$\sum_{i=0}^{k}(a_0+ai)=\frac{(k+1)(ak+2a_0)}{2}$$
together with the assumption $p>ak+2a_0$ yields that
$$v_p\Big(\sum_{i=0}^{k}(a_0+ai)\Big)=v_p(k+1)+v_p(ak+2a_0)-v_p(2)=0.$$
This infers that
$$v_p\bigg(\frac{\sum_{i=0}^{k}(a_0+ai)}{\prod_{i=0}^{k}(a_0+ai)}\bigg)=0.$$
Therefore $v_p(S_1)=-k$ as claimed. The claim is proved.

Let's now consider $v_p(S_2)$. Since
$$p^2>p\big(ak+2a+\frac{2b}{p}\big)>
\frac{n}{k+1}\cdot a(k+2)+p\cdot\frac{2b}{p}>a(n-1)+b,$$
it follows that $v_p(ai+b)\le 1$ for all $0\le i\le n-1$. Then
\begin{align*}
v_p(S_2)&=v_p\Big(\sum_{\substack{0\leq i_1<\cdots<i_k\leq n-1 \\
\exists j\ {\rm s.t.\ }p\nmid(ai_j+b)}}\prod_{j=1}^{k}\frac{1}{ai_j+b}\Big)\\
&\ge\min_{\substack{0\leq i_1<\cdots<i_k\leq n-1 \\
\exists j\ {\rm s.t.\ }p\nmid(1+ai_j)}}v_p\bigg(\prod_{j=1}^{k}\frac{1}{1+ai_j}\bigg)
\ge 1-k. \tag{2.14}
\end{align*}

Finally, by the above claim and (2.14), we can derive immediately that
$$v_p(S_{a,b}(n,k))=v_p(S_1+S_2)=-k$$
as required. This ends the proof of Lemma 2.5.
\end{proof}

\section{\bf Proof of Theorem 1.1}

This section is devoted to the proof of Theorem 1.1.

\textit{Proof of Theorem 1.1}.
Clearly $S_{a,b}(n,k)=1$ if $b=n=k=1$ and $S_{a,b}(n,k)$ is not an integer
if $n=k=1$ and $b\ge 2$. By Lemma 2.1 we know that $S_{a,b}(n,k)$
is not an integer if $k=1$ and $n\ge 2$. So we let $2\le k\le n$
in what follows.

First let $a>18$ or $b>\frac{3275a(\sqrt{2}e+1)}{e^a-1}+1$. If either
$2\le n\le\frac{b}{a}e^{a(\sqrt{2b^2+1}-1)/b}-\frac{b}{a}$ or
$k\ge\frac{e}{a}\log\frac{an+b}{b}+\frac{e}{b}$, then by Lemma 2.2, one has that
$0<S_{a,b}(n,k)<1$, which implies that $S_{a,b}(n,k)$ is not an integer.
If $n>\frac{b}{a}e^{a(\sqrt{2b^2+1}-1)/b}-\frac{b}{a}$ and
$k<\frac{e}{a}\log\frac{an+b}{b}+\frac{e}{b}$, then by Lemma 2.4
there is a prime $p$ satisfying $\frac{n}{k+1}<p\le \frac{n}{k}$
and $p>ak+2a+6$. Hence $\frac{b}{p}<\frac{b(k+1)}{n}$.
But from $n>\frac{b}{a}e^{a(\sqrt{2b^2+1}-1)/b}-\frac{b}{a}$ and
$k<\frac{e}{a}\log\frac{an+b}{b}+\frac{e}{b}$ one derives that
$$\frac{b(k+1)}{n}<\frac{b(\frac{e}{a}\log\frac{an+b}{b}+\frac{e}{b}+1)}{n}
=\frac{e\log(1+\frac{an}{b})}{\frac{an}{b}}+\frac{b+e}{n}<3.$$
So $p>ak+2a+6>ak+2a+2b/p$. It then follows from Lemma 2.5
that $v_p\big(S_{a,b}(n,k)\big)=-k<0$.
Thus $S_{a,b}(n,k)$ is not an integer if
$n>\frac{b}{a}e^{a(\sqrt{2b^2+1}-1)/b}-\frac{b}{a}$ and
$k<\frac{e}{a}\log\frac{an+b}{b}+\frac{e}{b}$. This concludes
that $S_{a,b}(n,k)$ is not an integer if $a>18$ or
$b>\frac{3275a(\sqrt{2}e+1)}{e^a-1}+1$.

Consequently, let $a\le 18, b\le\frac{3275a(\sqrt{2}e+1)}{e^a-1}+1$ and $n>120000$.
If $k\ge\frac{e}{a}\log\frac{an+b}{b}+\frac{e}{b}$, then by Lemma 2.2,
one has $0<S_{a,b}(n,k)<1$. If $k<\frac{e}{a}\log\frac{an+b}{b}+\frac{e}{b}$,
then by Lemma 2.4 there is a prime $p$ satisfying $\frac{n}{k+1}<p\le \frac{n}{k}$
and $p>ak+2a+6$. Hence $\frac{b}{p}<\frac{b(k+1)}{n}<3$, which gives that
$p>ak+2a+6>ak+2a+2b/p$. Then by Lemma 2.5 we obtain that
$v_p\big(S_{a,b}(n,k)\big)=-k<0$. So $S_{a,b}(n,k)$ is not an integer
if $a\le 18, b\le\frac{3275a(\sqrt{2}e+1)}{e^a-1}+1$ and $n>120000$.

By Lemma 2.2, it remains to prove that $S_{a,b}(n,k)$
is not an integer if $a\le 18$, $b\le\frac{3275a(\sqrt{2}e+1)}{e^a-1}+1$,
$2\le k<\frac{e}{a}\log\frac{an+b}{b}+\frac{e}{b}$
and $\frac{b}{a}\big(e^{a(\sqrt{2b^2+1}-1)/b}-1\big)<n\le 120000$. This
will be done in what follows.

Before doing so, we need to develop an analysis about prime distribution in the intervals
$(\frac{n}{k+1}, \frac{n}{k}]$ where $2\le k<e\log120001+e<35$ and $n\le 120000$.
Let $p_i$ denote the $i$-th prime. For $2\le k\le 34$, define $i_k$ to be the integer
satisfying that $kp_{i_k}\ge(k+1)p_{i_k-1}$ and $kp_{i+1}<(k+1)p_i$ for all integers $i$
with $i_k\le i\le 11301$, where $p_{11301}=119993$, the biggest prime less than 120000.
We list all the values of $i_k$ and $p_{i_k}$ in the following Table 1.
Evidently, Table 1 gives us the observation that $(k+1)/p_{i_k}<1/2$ for $2\le k\le 34$.
\begin{table}[h]\caption{}\begin{tabular}{|c|c|c|c|c|c|c|c|c|c|c|c|}
  \hline
  % after \\: \hline or \cline{col1-col2} \cline{col3-col4} ...
       $k$  & 2  & 3  & 4  & 5  & 6  & 7  & 8  & 9  & 10 & 11 & 12      \\\hline
     $i_k$  & 5  & 5  & 10 & 10 & 12 & 12 & 16 & 31 & 31 & 31 & 31      \\\hline
  $p_{i_k}$ & 11 & 11 & 29 & 29 & 37 & 37 & 53 & 127& 127& 127& 127     \\\hline
       $k$  & 13 & 14 & 15 & 16 & 17 & 18 & 19 & 20 & 21 & 22 & 23      \\\hline
     $i_k$  & 31 & 35 & 35 & 35 & 47 & 48 & 48 & 48 & 63 & 63 & 67      \\\hline
  $p_{i_k}$ & 127& 149& 149& 149& 211& 223& 223& 223& 307& 307& 331     \\\hline
       $k$  & 24 & 25 & 26 & 27 & 28 & 29 & 30 & 31 & 32 & 33 & 34      \\\hline
     $i_k$  & 67 & 67 & 67 & 67 & 67 & 67 & 100& 100& 100& 100& 100     \\\hline
  $p_{i_k}$ & 331& 331& 331& 331& 331& 331& 541& 541& 541& 541& 541     \\\hline
\end{tabular}\end{table}
We claim that for any integer $k$ with $2\le k\le 34$, if $kp_{i_k}\le n\le 120000$,
then there is always a prime $p$ such that $p\ge p_{i_k}$ and $\frac{n}{k+1}<p\le\frac{n}{k}$.
Actually, if $p_{i_k}>\frac{n}{k+1}$, then we have done. If $p_{i_k}\le \frac{n}{k+1}$,
then the fact $\frac{n}{k+1}<120000$ tells us that
there is an index $i$ with $i_k\le i<11301$ such that $p_i\le\frac{n}{k+1}<p_{i+1}$.
Since $kp_{i+1}<(k+1)p_i\le n$, we have $p_{i+1}<\frac{n}{k}$. So letting
$p:=p_{i+1}$ givers us the desired result $\frac{n}{k+1}<p<\frac{n}{k}$ and the claim is proved.

Let's continue the proof of Theorem 1.1. Let $13\le a\le 18$. Then one can easily check
that $b\le\frac{3275a(\sqrt2e+1)}{e^a-1}+1<2$. So $b=1$. Furthermore we have
$k<\frac{e}{a}\log\frac{an+b}{b}+\frac{e}{b}\le\frac{e}{a}\log 120000a+e<6$.
It then follows that
$n>\frac{b}{a}e^{a(\sqrt{2b^2+1}-1)/b}-\frac{b}{a}=\frac{e^{(\sqrt3-1)a}-1}{a}>kp_{i_k}$.
So the above claim infers that there is a prime $p$ such that
$\frac{n}{k+1}<p\le\frac{n}{k}$. Besides, since $k<6$, one has
$p>n/(k+1)\ge\frac{e^{(\sqrt3-1)a}-1}{6a}>7a+2>ak+2a+2b/p$ for all integers $a$
with $13\le a\le 18$. Then applying Lemma 2.5 yields $v_p\big(S_{a,b}(n,k)\big)=-k<0$.
Hence $S_{a,b}(n,k)$ is not an integer in this case.

Let $2\le a\le 12$ and $b\le\min\{27,\frac{3275a(\sqrt2e+1)}{e^a-1}+1\}$.
Then $k<\frac{e}{a}\log\frac{an+b}{b}+\frac{e}{b}\le \frac{e}{a}\log(120000a+1)+e
\le\frac{e}{2}\log 240001+e<20$. Define $k_a:=\big\lfloor\frac{e}{a}\log(120000a+1)+e\big\rfloor$
and
$n_a:=\max\big\{k_ap_{i_{k_a}}, a(k_a+1)(k_a+2)
+\big\lceil\frac{2b(k_a+1)}{p_{i_{k_a}}}\big\rceil\big\}.$
Then the value of $n_a$ for $2\le a\le 12$ can be listed as follows:
\begin{center}\begin{tabular}{|c|c|c|c|c|c|c|c|c|c|c|c|c|}
  \hline
  % after \\: \hline or \cline{col1-col2} \cline{col3-col4} ...
   $a$  &  2    & 3    & 4    & 5    & 6   & 7   & 8   & 9   & 10  & 11  & 12   \\\hline
  $n_a$ &  4437 & 2086 & 1397 & 1143 & 550 & 640 & 588 & 515 & 571 & 627 & 516  \\
  \hline
\end{tabular}\end{center}
Now fix an integer $a$ with $2\le a\le 12$. If $n_a\le n\le 120000$ and $2\le k\le k_a$,
then $n\ge k_ap_{i_{k_a}}\ge kp_{i_k}$. It follows that $p_{i_{k_a}}\le\frac{n}{k_a}\le\frac{n}{k}$.
If $p_{i_{k_a}}\le\frac{n}{k+1}$, then by the above claim we know that there is a prime $p'$
satisfying that $\frac{n}{k+1}<p'\le\frac{n}{k}$. Clearly, $p'>p_{i_{k_a}}$.
If $p_{i_{k_a}}>\frac{n}{k+1}$, then $\frac{n}{k+1}<p_{i_{k_a}}\le \frac{n}{k}$.
This concludes that we can always choose a prime $p\ge p_{i_{k_a}}$ such that
$\frac{n}{k+1}<p\le \frac{n}{k}$. For such prime $p$, we have
$p>\frac{n}{k+1}\ge\frac{n_a}{k_a+1}\ge ak+2a+2b/p$.
Hence by Lemma 2.5, $S_{a,b}(n,k)$ is not an integer. If $n\le n_a-1$ and $k\le k_a$,
then by direct computations using Maple 12 (see Program 1 in Appendix) and the following
recursive formulas:
$$S_{a,b}(1,1)=\frac{1}{b},\ S_{a,b}(n,1)=S_{a,b}(n-1,1)+\frac{1}{b+(n-1)a}, \eqno (3.1)$$
$$S_{a,b}(n,k)=S_{a,b}(n-1,k)+\frac{1}{b+(n-1)a}S_{a,b}(n-1,k-1) \ \text{for} \ 2\le k\le n-1 \eqno (3.2)$$
and
$$S_{a,b}(n, n)=\frac{1}{b+(n-1)a}S_{a, b}(n-1, n-1), \eqno (3.3)$$
we can check
that $S_{a,b}(n,k)$ is not an integer in this case.

Since $\frac{3275a(\sqrt2e+1)}{e^a-1}+1<28$ if $9\le a\le 12$, the above
proof implies that $S_{a,b}(n,k)$ is not an integer if $9\le a\le 12$.
Hence to complete the proof for the case $a\ge 2$, one may
let $2\le a\le 8$ and $28\le b\le\frac{3275a(\sqrt2e+1)}{e^a-1}+1$.
Then $k<\frac{e}{a}\log\frac{an+b}{b}+\frac{e}{b}\le\frac{e}{a}
\log(\frac{120000a}{28}+1)+\frac{e}{28}<13$. Since
$k<\frac{e}{a}\log\frac{an+b}{b}+\frac{e}{b}$, we have
$n>\frac{b}{a}\big(e^{a(k/e-1/b)}-1\big)$. We can check that
$\frac{b}{a}\big(e^{a(k/e-1/b)}-1\big)\ge\frac{28}{2}\big(e^{2(k/e-1/28)}-1\big)
>kp_{i_k}$ for each $k$ with $2\le k\le 12$. Hence by the above claim we know that
there is a prime $p\ge p_{i_k}$ satisfying $\frac{n}{k+1}<p\le\frac{n}{k}$.
Since $p\ge p_{i_k}$, by the above observation, one has $(k+1)/p\le (k+1)/p_{i_k}<1/2$.
It then follows that
\begin{align*}
n-(k+1)&\Big(ak+2a+\frac{2b}{p}\Big)
>b\Big(\frac{e^{a(k/e-1/b)}}{a}-\frac{1}{a}-1\Big)-a(k+1)(k+2) \\
&>28\Big(\frac{e^{a(k/e-1/28)}}{2}-\frac{1}{a}-1\Big)-a(k+1)(k+2). \tag{3.4}
\end{align*}
We can easily check that the right-hand side of (3.4) is positive if $2\le a\le 8$ and
$2\le k\le 12$. Thus $p>\frac{n}{k+1}>ak+2a+2b/p$. Therefore by Lemma 2.5,
$S_{a.b}(n,k)$ is not an integer if $2\le a\le 8$ and
$28\le b\le\frac{3275a(\sqrt2e+1)}{e^a-1}+1$. This concludes that
$S_{a,b}(n,k)$ is not an integer if $a\ge 2$.
To finish the proof of Theorem 1.1, one needs only to handle
the remaining case $a=1$. In the following we let $a=1$.

Let $b\ge 45$. Then $k<e\log\frac{n+b}{b}+\frac{e}{b}\le
e\log\frac{120045}{45}+\frac{e}{45}<25$ and $n>b(e^{k/e-1/b}-1)$.
So for any integer $k$ with $2\le k\le 24$, we have
$n>b\big(e^{k/e-1/b}-1\big)\ge45\big(e^{k/e-1/45}-1\big)>kp_{i_k}$ and
$n-(k+1)(ak+2a+2b/p_{i_k})>b\big(e^{k/e-1/b}-1-2(k+1)/p_{i_k}\big)-(k+1)(k+2)>0$.
Hence by the above claim, there is a prime $p\ge p_{i_k}$ with $\frac{n}{k+1}<p\le\frac{n}{k}$.
For such a prime $p$, one has $p>\frac{n}{k+1}>ak+2a+2b/p_{i_k}>ak+2a+2b/p$,
which implies that $S_{a,b}(n,k)$ is not an integer by Lemma 2.5.

Let $b\le 44$ and $k\ge 24$. If $b=1$, then by \cite{[CT]}, $S_{1,1}(n,k)$ is not an
integer. If $2\le b\le44$, then $k<e\log\frac{n+b}{b}+\frac{e}{b}\le\log60001+e/2<32$
and $n>b(e^{k/e-1/b}-1)$. It is easy to check for any $k$ with $24\le k\le 31$ that
$b(e^{k/e-1/b}-1)>2(e^{k/e-1/2}-1)>kp_{i_k}$ and
$b(e^{k/e-1/b}-1)-(k+2)(k+1)-2b(k+1)/{p_{i_k}}>2(e^{k/e-1/b}-1-2(k+1)/p_{i_k})-(k+2)(k+1)>0$.
So by the above claim there is a prime $p$ such that
$\frac{n}{k}\ge p>\frac{n}{k+1}>\frac{b(e^{k/e-1/b}-1)}{k+1}>k+2+2b/p$.
Hence by Lemma 2.5, $S_{a,b}(n,k)$ is not an integer in this case.

Let $b\le 44$ and $2\le k\le 23$. If $n\ge 7613$, i.e., $n\ge 7613=23p_{i_{23}}$,
then $n\ge kp_{i_k}$ for $2\le k\le 23$. It follows from the above claim
that there is a prime $p$ satisfying $\frac{n}{k+1}<p\le \frac{n}{k}$. Further,
one has $n\ge7613>(k+1)(k+2+b)$ for any integer $k$ with $2\le k\le 23$, and so
$p>\frac{n}{k+1}>k+2+b\ge k+2+2b/p$. Therefore one yields from Lemma 2.5 that
$S_{a,b}(n,k)$ is not an integer. If $n\le 7612$, then using Maple 12
(see Program 2 in Appendix) and the recursive formulas (3.1) to (3.3),
one can check that $S_{a,b}(n,k)$ is not an integer except that $b=1, n=3$
and $k=2$, in which case $S_{a,b}(n,k)=1$.

This completes the proof of Theorem 1.1.

\section*{Appendix}

\noindent{\bf Program 1.}
\begin{verbatim}
IntTest1:=proc(a) local i,j,C,b,k,n,S;
C:=min(27, floor(3275*a*(1.4143*2.7183+1)/(2.7182^a-1)+1));
k:=floor(2.7183*log[2.7182](120000*a+1)/a+2.7183); S:=vector(k,0);
n:=[0,4437,2086,1397,1143,550,640,588,515,571,627,516];
for b from 1 to C do S:=vector(k,0);
for i from k to n[a] do S[1]:=S[1]+1/(a*i-a*k+b);
for j from 2 to k do S[j]:= S[j]+S[j-1]/(a*i-a*k+a*j-a+b);
if type(S[j],integer) then print(i-k+j, j*IsInt) end if;
end do; end do; end do; end proc
for a from 2 to 12 do IntTest1(a) end do;
\end{verbatim}

\noindent{\bf Program 2.}
\begin{verbatim}
IntTest2:=proc(b)
local i,j,S; S:=vector(23,0);
for i from 23 to 7612 do S[1]:=S[1]+1/(i-23+b);
for j from 2 to 23 do S[j]:= S[j]+S[j-1]/(i+j-24+b);
if type(S[j],integer) then print(i-23+j, j*IsInt) end if;
end do; end do; end proc
for b from 1 to 44 do IntTest2(b) end do
\end{verbatim}

%-------------------------------------------------------------------
\end{document}